\documentclass[12pt]{amsart}
\usepackage{amsmath,amssymb,amsfonts,amsthm,amsopn}
\usepackage{pb-diagram}
\newcommand{\tfa}{time-frequency analysis}

\newcommand{\ft}{Fourier transform}
\newcommand{\stft}{short-time Fourier transform}

\newcommand{\tf}{time-frequency}

\newtheorem{theorem}{Theorem}[section]

\newtheorem{corollary}[theorem]{Corollary}
\newtheorem{proposition}[theorem]{Proposition}
\newtheorem{definition}[theorem]{Definition}

\newtheorem{remark}[theorem]{Remark}

\theoremstyle{definition}

\newcommand{\beqa}{\begin{eqnarray*}}
\newcommand{\eeqa}{\end{eqnarray*}}

\DeclareMathOperator*{\essupp}{ess\,sup\,}

\newcommand{\field}[1]{\mathbb{#1}}
\newcommand{\bR}{\field{R}}        
\newcommand{\bN}{\field{N}}        
\newcommand{\bZ}{\field{Z}}        
\newcommand{\bC}{\field{C}}        
\def\la{\lambda}

\def\th{\theta}

\def\cF{\mathcal{F}}              
\def\cS{\mathcal{S}}

\def\a{\alpha}

\def\rd{\bR^d}

\def\rdd{{\bR^{2d}}}
\def\zdd{{\bZ^{2d}}}

\def\lrd{L^2(\rd)}
\def\lrdd{L^2(\rdd)}

\def\intrd{\int_{\rd}}
\def\intrdd{\int_{\rdd}}

\def\<{\left<}
\def\>{\Big)>}

\def\mv1{M_v^1}

\def\phas{(x,\xi)}
\def\mn{(m,n)}
\def\mn'{(m',n')}

\def\o{\xi}
\def\a{\alpha}

\def\Ren{\mathbb{R}^d}

\def\sch{\mathcal{S}}

\def\Tau{\mathcal{T}}
\def\tauhz0{\widehat{\mathcal{T}}^\hbar(z_0)}
\def\tauhz{\widehat{\mathcal{T}}^\hbar(z)}

\def\Sn2{S_{2}(L^{2}(\Ren))}
\def\S1{S_{1}(L^{2}(\Ren))}
\def\sig00{\sigma_{0,0}}
\def\th{\widehat{\mathcal{T}}^\hbar}

\def\la{\langle}
\def\ra{\rangle}

\begin{document}
\title {Kernel Theorems for modulation spaces}
\author{Elena Cordero}
\address{Universit\`a di Torino, Dipartimento di
Matematica, via Carlo Alberto 10, 10123 Torino, Italy}
\email{elena.cordero@unito.it}
\author{Fabio Nicola}
\address{Dipartimento di Scienze Matematiche, Politecnico di Torino, corso
Duca degli Abruzzi 24, 10129 Torino, Italy}
\email{fabio.nicola@polito.it}
\date{}

\begin{abstract}
We deal with kernel theorems for modulation spaces. We completely characterize the continuity of a linear operator on the modulation spaces $M^p$ for every $1\leq p\leq\infty$, by the membership of its kernel to (mixed) modulation spaces. Whereas Feichtinger's kernel theorem (which we recapture as a special case) is the modulation space counterpart of  Schwartz' kernel theorem for temperate distributions, our results do not have a couterpart in distribution theory. This reveals the superiority, in some respects, of the modulation space formalism upon distribution theory, as already emphasized in Feichtinger's manifesto for a post-modern harmonic analysis, tailored to the needs of mathematical signal processing. The proof uses in an essential way a discretization of the problem by means of Gabor frames. We also show the equivalence of the operator norm and the modulation space norm of the corresponding kernel. For operators acting on $M^{p,q}$ a similar characterization is not expected, but sufficient conditions for boundedness can be sated in the same spirit.
\end{abstract}

\maketitle

\section{Introduction}
Schwartz' kernel theorem certainly represents one of the most important results in modern Functional Analysis. It states, in the framework of temperate distributions, that every linear continuous operator $A:\cS(\rd)\to\cS'(\rd)$ can be regarded as an integral operator in a generalized sense, namely 
\[
\langle Af,g\rangle=\langle K,g\otimes \overline{f}\rangle\qquad f,g\in\cS(\rd),
\]
in the sense of distributions, for some kernel $K\in\cS'(\rdd)$, and viceversa \cite{hor}. We write, formally,
\begin{equation}\label{IOP}
A f(x)=\intrd K(x,y)\,f(y)dy\quad f\in\cS(\rd).
\end{equation}
This result provides a general big framework including most operators of interest in Harmonic Analysis. However, for most applications in Time-frequency Analysis and Mathematical Signal Processing the spaces $\cS(\rd)$ and $\cS'(\rd)$ can be fruitfully replaced by the modulation spaces $M^1(\rd)$ and $M^\infty(\rd)$, respectively, introduced by H. Feichtinger in \cite{F1} and nowadays widely used as a fundamental tool in Time-frequency Analysis (see \cite{book} and also Feichtinger's manifesto for a postmodern Harmonic Analysis \cite{fei}). One of the obvious advantages is that $M^1(\rd)$ and $M^\infty(\rd)$ are Banach spaces; secondly  their norm gives a direct and transparent information on the time-frequency concentration of a function or temperate distribution. The same holds, more generally, for a whole scale of intermediate spaces $M^p(\rd)$, $1\leq\ p\leq\infty$, as well as the more general mixed norm spaces $M^{p,q}(\rd)$, $1\leq p,q\leq\infty$. \par
In short, for $(x,\xi)\in\rd\times\rd$, define the time-frequency shifts $\pi(x,\xi)$ by
\[
\pi(x,\xi)f(t)=e^{2\pi i t \xi} f(t-x), \quad f\in\cS'(\rd).
\]
Then $M^{p,q}(\rd)$ is the space of temperate distributions $f$ in $\rd$ such 
that the function 
\[
(x,\xi)\mapsto\langle f,\pi(x,\xi)g\rangle,
\]
for some (and therefore any) window $g\in\cS(\rd)\setminus\{0\}$, is in $L^q(\rd_\xi;L^p(\rd_x))$, endowed with the obvious norm (see Section 2 below for details). We also set $M^p(\rd)=M^{p,p}(\rd)$.
Weighted versions are also used in the literature but here we limit ourselves, for simplicity, to the unweighted case. \par
Now, a kernel theorem in the framework of modulation spaces was  announced in \cite{fei0} and proved in \cite{gei1}; see also \cite[Theorem 14.4.1]{book}. It states that linear continuous operators $A:M^1(\rd)\to M^\infty(\rd)$ are characterized by the memenbership of thier distribution kernel $K$ to $M^\infty(\rdd)$.\par
In this note we provide a similar characterization of linear continuous operators acting on the following spaces:\par\smallskip
\begin{itemize}
\item $M^1(\rd)\to M^p(\rd)$, for a fixed $1\leq p\leq\infty$
\item $M^p(\rd)\to M^\infty(\rd)$, for a fixed $1\leq p\leq\infty$
\item $M^p(\rd)\to M^p(\rd)$, for every $1\leq p\leq\infty$.
\end{itemize}\par\smallskip
In all cases the characterization is given in terms of the membership of their distribution kernels to certain {\it mixed modulation spaces}, first introduced and studied in \cite{Bi2010}. The proof uses in an essential way a discretization of the  problem by means of Gabor frames. We also show the equivalence of the operator norm and the modulation space norm of the corresponding kernel. \par
These results seem remarkable, because no such a characterization is known, for example, for linear continuous operators from $\cS(\rd)$ into itself and similarly from $\cS'(\rd)$ into itself. This again shows the superiority, in some respects, of the modulation space framework upon distribution theory.\par
For operators acting on mixed-norm modulation spaces $M^{p,q}(\rd)\to M^{p,q}(\rd)$ such a characterization is not expected, as we will see in Section 3 below. We have however sufficient conditions for boundedness in the same spirit.\par
Let us also observe that in the literature there are plenty of {\it sufficient conditions} for operators in certain classes (Fourier multipliers, localization operators, pseudodifferential operators, Fourier integral operators) to be bounded $M^p(\rd)\to M^p(\rd)$ for every $1\leq p\leq\infty$, or $M^{p,q}(\rd)\to M^{p,q}(\rd)$ for every $1\leq p,q\leq\infty$. For example it is known that pseudodifferential operators with Weyl symbol in $M^{\infty,1}$ are in fact bounded $M^{p,q}(\rd)\to M^{p,q}(\rd)$ for every $1\leq p,q\leq\infty$ \cite{co1,book,gro1,toft1,Toftweight} (see also \cite{Sjostrand1}). On the other hand the Fourier multipier with symbol $e^{i|\xi|^2}$ is still bounded on these spaces \cite{benyi}, although that symbol does not belong to $M^{\infty,1}(\rdd)$. Similary, it is easy to see that any linear continuous operator $A:\cS(\rd)\to\cS'(\rd)$ whose Gabor matrix enjoys decay estimates such as
\[
|\langle A\pi(z)g,\pi(w)g\rangle|\leq C(1+|w-\chi(z)|)^{-s}\qquad z,w\in\rdd
\]
for some $s>2d$, $g\in\cS(\rd)\setminus\{0\}$, and bi-Lipschtiz mapping $\chi:\rdd\to\rdd$, is bounded $M^p(\rd)\to M^p(\rd)$ for every $1\leq p\leq\infty$ \cite{fio3} (see also \cite{fio3,fio1,fio2}). But once again this condition is not necessary at all. \par
The very simple characterization provided in the present paper when $p=q$ therefore includes, at least implicitly, all these continuity results on $M^p$.

\section{Backgrounds on \tfa} 
\subsection{Mixed norm spaces} We summarize the main properties of mixed norm spaces, we refer to \cite{BP61} for a full treatment of this subject. \par

Consider measure spaces $(X_i,\mu_i)$, indices $p_i\in [1,\infty]$, with $i=1,\dots,d$.  Then the mixed norm space 
$L^{p_1,p_2,\dots, p_d}(X_1\times X_2\times\dots\times X_d;\mu_1\times\mu_2\times\dots\times\mu_d)$ consists of the measurable functions $F:X_1\times X_2\times\dots\times X_d\to \bC$  such that the following norm is finite:
\begin{equation}\label{MNS}
\|F\|_{L^{p_1,p_2,\dots, p_d}}=\left(\int_{X_d}\cdots\left(\int_{X_1}|F(x_1,\dots,x_d)|^{p_1 }\,d\mu_{1}(x_1)\right)^{\frac{p_2}{p_1}}\cdots \,d\mu_{d}(x_d)\right)^{\frac{1}{p_d}}
\end{equation}
with standard modification when $p_i=\infty$ for some $i$.\par
If $X_i=\bR$ and $\mu_i$ is the Lebesgue measure on $\bR$ for every $i=1,\dots,d$, we simply write $L^{p_1,\dots, p_d}$.
If each $X_i$ is a countable set and the corresponding measure $\mu_i$ is the counting measure, we use the notation
$\ell^{p_1,\dots, p_d}(X_1\times X_2\times\dots\times X_d)$.

The mixed norm spaces $L^{p_1,p_2,\dots, p_d}(X_1\times X_2\times\dots\times X_d;\mu_1\times\mu_2\times\dots\times\mu_d)$ are Banach spaces, generalizations of the classical $L^p$ and $\ell^p$ spaces, cf. \cite{BP61}. 
   \subsection{Modulation spaces}
 For $x,\xi\in \rd$, define the translation operator $T_x$ and modulation operator $M_\xi$ by
 $$T_x f(t)=f(t-x)\quad M_\xi f(t)= e^{2\pi i tx} f(t).
 $$
For $z=\phas$, the composition operator $\pi(z)=M_\xi T_x$ is called a \tf \, shift of the phase space $\rdd$. 

 For a fixed non-zero $g \in \cS (\rd )$, the \stft\ (STFT) of $f \in
 \cS ' (\rd ) $ with respect to the window $g$ is given by
 \begin{equation}
 \label{eqi2}
 V_gf(x,\o)=\int_{\rd}
 f(t)\, {\overline {g(t-x)}} \, e^{-2\pi i\o t}\,dt\, =\la f, \pi(z)g\ra\quad z=\phas,
 \end{equation}
where the integral is intended as the (anti-)duality between $\cS'$ and $\cS$.  The \stft\, gives information about the \tf\, content of the signal $f$: indeed, roughly speaking, it can be seen as a localized \ft\, of $f$  in a neighbourhood of $x$.

The STFT is the basic tool in the definition of modulation spaces. Indeed,  modulation  space norms are a  measure
of
the joint time-frequency distribution of $f\in \sch '$. For their
basic properties we refer  \cite{feichtinger80,feichtinger83,feichtinger90} (see also \cite{Bi2010, MOP2016, MP2013, Pfander2013}) and the textbook \cite{book}.
\begin{definition}\label{2.1}
Fix $g\in \cS(\rd)$, $p=(p_1,\dots,p_d)$, $q=(q_1,\dots,q_d)$, $p_i,q_i\in [1,\infty]$, $i=1,\dots,d$. Then 
$$M^{p,q}(\rd)=\{f\in\cS'(\rd):\, \|f\|_{M^{p,q}(\rd)}<\infty\},
$$
where 
$$\|f\|_{M^{p,q}(\rd)}=\|V_g f\|_{L^{p_1,\dots,p_d,q_1,\dots,q_d}}.
$$
\end{definition}
If 
$p=p_1,\dots,p_d,\quad q=q_1,\dots,q_d$, we come back to the classical modulation spaces of distributions $f\in \cS'(\rd)$ such that 
$$\|f\|_{M^{p,q}(\rd)}=\left(\intrd \left( \intrd |V_g f \phas|^p\,dx\right)^{\frac q p} d\xi \right)^\frac 1 q <\infty$$
(cf. \cite{F1}). The extension of the original modulation spaces in Definition \ref{2.1} was widely studied and applied to the investigation of boundedness and sampling properties for  pseudodifferential operators in \cite{MOP2016, MP2013, Pfander2013}. \par 
To state our kernel theorems we need a further extension of Definition \ref{2.1}, that was  introduced by S. Bishop in \cite{Bi2010}, in her study of Schatten $p$-class properties for integral operators.  This is a natural generalization of the classical modulation spaces, as will be clear in the sequel.

We shall mainly use the notation in \cite{Bi2010}. From now on we assume that $c$ is a permutation of the set $\{1,\dots,2d\}$.  We identify $c$ with the linear bijection  
$$\tilde{c}:\rdd \to\rdd,\quad \tilde{c}(x_1,\dots,x_{2d})=(x_{c(1)},\dots, x_{c(2d)}).
$$

\begin{definition}\label{2.2}
	Consider $g\in\cS(\rd)$ and let $c$ be a permutation corresponding to the map $\tilde{c}$ as above. Then  $M(c)^{p_1,p_2,\dots,p_{2d}}$ is the \emph{\bf mixed modulation space} of tempered distributions $f\in\cS'(\rd)$ for which
	\begin{equation}\label{norm-mix}
\|f\|_{M(c)^{p_1,p_2,\dots,p_{2d}}} =\|V_g f\circ \tilde{c}\|_{L^{p_1,p_2,\dots,p_{2d}}} <\infty.
	\end{equation}
\end{definition}
If $p=p_1=p_2=\dots=p_d,\quad q=p_{d+1}=p_{d+2}=\dots=p_{2d}
$ we simply write $M(c)^{p,q}$; if $p=p_1=p_2=\dots=p_d=p_{d+1}=p_{d+2}=\dots=p_{2d}$, we shorten to $M(c)^{p}$.
\begin{remark}\ \par
	\begin{itemize}
		\item [(i)] If $c$ is the identity permutation and the indices satisfy
		$$p=p_1=p_2=\dots=p_d,\quad q=p_{d+1}=p_{d+2}=\dots=p_{2d}
		$$
		then $M(c)^{p_1,p_2,\dots,p_{2d}}=M^{p,q}(\rd)$ is the original modulation space introduced by Feichtinger in \cite{F1}.
			\item [(ii)] If $c$ is the identity permutation and $p=(p_1,\dots,p_d),\, q=(q_1,\dots,q_d)$, we come back to the modulation spaces in Definition \ref{2.1}.
			\item [(iii)] If the indices satisfy
			$$p=p_1=p_2=\dots=p_d=p_{d+1}=p_{d+2}=\dots=p_{2d},
			$$
			given any permutation $c$, a simple change of variables yields
			$$ M(c)^{p}=M^p(\rd).
			$$
		\item[(iv)] In general  the  mixed modulation spaces  differ from the classical ones of Definition \ref{2.1}. 
A simple example is provided by the permutation 
\begin{equation}\label{perminv}
c_0(1,2,\dots,d,d+1,\dots,2d)=(d+1,d+2,\dots,2d,1,2,\dots, d).
\end{equation}
If we choose
$$p=p_1=p_2=\dots=p_d,\quad q=p_{d+1}=p_{d+2}=\dots=p_{2d}
$$
then it is easy to see that
$$ M^{p,q}(c_0)=W(\cF L^p,L^q)(\rd),
$$
that is the space of distributions $f\in\cS'(\rd)$ such that 
$$\|f\|_{W(\cF L^p,L^q)(\rd)}=\left(\intrd \left( \intrd |V_g f \phas|^p\,d\xi\right)^{\frac q p} dx \right)^\frac 1 q <\infty$$
for some window function $g\in\cS(\rd)\setminus\{0\}.$

	\end{itemize}
\end{remark}

It can be shown (cf.\ \cite{Bi2010}) that the definition of mixed modulation spaces is independent of the choices of the window $g$ in $\cS(\rd)$, with different windows giving equivalent norms. Furthermore, this fact also holds for $g$ in the larger space $M^1(\rd)$, as follows by the subsequent Theorem \ref{Invform}, which gives the inversion formula on $M(c)^{p_1,\dots, p_{2d}}$.
\begin{theorem}\label{Invform}
Suppose $c$ is a permutation of $\{1,\dots,2d\}$, and for a measurable function $\psi:\rd \to \bC$ define the operator $\Upsilon_\psi$ by
\begin{equation}
\Upsilon_\psi F(t)=\intrdd F(x) \pi(\tilde{c}(x))\psi(t)\, dx. 
\end{equation}
(i) Given $\psi, g\in M^1(\rd)$, $p_1,\dots, p_{2d}\in [1,\infty]$, and $f\in  M(c)^{p_1,\dots, p_{2d}}$, we have 
$$\Upsilon_\psi (V_g f\circ \tilde{c})=\la\psi, g\ra f.
$$
(ii) Different windows in $M^1(\rd)$ define equivalent norms on  
	$ M(c)^{p_1,\dots, p_{2d}}$.
\end{theorem}

\begin{corollary}
	For any $p_1,\dots, p_{2d}\in [1,\infty]$, $M(c)^{p_1,\dots, p_{2d}}$ is a Banach space.
\end{corollary}
We also have the expected duality result.
\begin{theorem}
	If $p_1,\dots, p_{2d}\in [1,\infty)$, $M(c)^{p'_1,\dots, p'_{2d}}$ is the dual space of $M(c)^{p_1,\dots, p_{2d}}$, where $p'_i\in (1,\infty]$ satisfies $1/p_i+1/p_i'=1$.
\end{theorem}
\subsection{Gabor frames} Given a lattice $\Lambda$ in $\rdd$ and a non-zero square integrable function $g$ (called window) on $\rd$   the system
$$\mathcal{G}(g,\Lambda)=\{\pi(\lambda)\,:\,\lambda\in\Lambda\}
$$
is called a Gabor frame if it is a frame for $\lrd$, that is there exist constants $0<A\le  B$ such that
\begin{equation}\label{framedef}
A\|f\|_2^2\leq\sum_{\lambda\in \Lambda}|\la f,\pi(\lambda) g\ra|^2\leq B\|f\|_2^2,\quad \forall f \in\lrd.
\end{equation} 
If $\eqref{framedef}$ holds,
then there exists a  $\gamma\in\lrd$ (so-called dual window),
such that $\mathcal{G}(\gamma,\Lambda)$
is a frame for $\lrd$ and every $f\in\lrd$ can be expanded as
\begin{equation}\label{GabExp}
f=\sum_{\lambda\in\Lambda}\la f, \pi(\lambda)g \ra \pi(\lambda)\gamma= \sum_{\lambda\in\Lambda}\la f,\pi(\lambda)\gamma\ra \pi(\lambda) g,
\end{equation}
with unconditional convergence in $\lrd$.\par 
From now on we shall work with lattices $\Lambda=\a \bZ^{2d}$, for  suitable $\a>0$. 

If the window function belongs to the Feichtinger's algebra $M^1(\rd)$, then a Gabor frame for $\lrd$ is a frame also for mixed modulation spaces (cf. \cite[Theorem 4.6]{Bi2010}). 
\begin{theorem}\label{GM}
	Let $g$ be a window function in $M^1(\rd)$ such that $\mathcal{G}(g,\Lambda)$ is a frame for $L^2(\rd)$ with a dual frame $\mathcal{G}(\gamma,\Lambda)$. Consider $p_1,\dots,p_{2d} \in [1,\infty]$. Then 
	\begin{itemize}
	\item[(i)] There exist constants $0<A\le  B$ independent of $p_1,\dots,p_{2d}$ such that
	$$A\|f\|_{M(c)^{p_1,\dots, p_{2d}}}\leq \| V_g f\circ {\tilde{c}}|_{\Lambda}\|_{\ell^{p_1,\dots,p_{2d}}}\leq B\|f\|_{M(c)^{p_1,\dots, p_{2d}}},\quad \forall f \in{M(c)^{p_1,\dots, p_{2d}}}.
	$$
	\item[(ii)] We have
\begin{equation}\label{GabExp2}
f=\sum_{\lambda\in\Lambda}\la f, \pi(\lambda)g \ra \pi(\lambda)\gamma= \sum_{\lambda\in\Lambda}\la f,\pi(\lambda)\gamma\ra \pi(\lambda) g,
\end{equation}
with unconditional convergence in $M(c)^{p_1,\dots, p_{2d}}$ if 
$p_1,\dots,p_{2d}\in [1,\infty)$, weak* convergence in $M^\infty(\rd)$ if $p_1,\dots,p_{2d}\in [1,\infty]$.
	\end{itemize}
\end{theorem}
This generalizes the basic fact, valid for the identity permutation, that the coefficient operator $C_g$ and the reconstruction operator $D_\gamma$, defined by
 \begin{equation}\label{DC0}
 (C_g f)_\lambda:=\la f, \pi(\lambda)g\ra \qquad
  D_\gamma c:=\sum_{\lambda\in \Lambda} c_\lambda \pi(\lambda)\gamma,
  \end{equation} satisfy  
  \begin{equation}\label{DC}
    C_g\,: M^{p,q}(\rd)\, \to \ell^{p,q}(\Lambda), \quad   D_\gamma:\, \ell^{p,q}(\Lambda) \to M^{p,q}(\rd),\end{equation}
  with $D_\gamma C_g=I$, the identity on $M^{p,q}(\rd)$, cf.\ \eqref{GabExp2}, and similarly if we interchange the windows $g$ and $\gamma$.
	 \section{Kernel Theorems}
In this section we will present the characterization of continuity properties on  modulation spaces of integral  operators \eqref{IOP}.\par
We need the following result (see e.g.\ Tao's Lecture Notes \cite[Proposition 5.2]{Tao}).\begin{proposition}\label{Prop5.2}
	Assume that $(X,\nu_X)$, $(Y,\mu_Y)$ are measure space and $p\in [1,\infty]$. Suppose that $K:X\times Y\to \bC$, and   $\|K(\cdot,y)\|_{L^p(X)} $ is uniformly bounded on $Y$.
	Then the operator $A$ with kernel $K$  maps $L^1(X)$ into $L^p(Y)$ and
	\begin{equation}\label{normk}
	\|A\|_{L^1(X)\to L^p(Y)}=\essupp_{y\in Y} \|K(\cdot,y)\|_{L^p(X)}.
	\end{equation}
\end{proposition}
Similarly, we have a dual statment (see e.g.\ \cite[Proposition 5.4]{Tao}).
\begin{proposition}\label{Prop5.4}
	Assume that $(X,\nu_X)$, $(Y,\mu_Y)$ are measure space and $p\in [1,\infty]$. Suppose that $K:X\times Y\to \bC$, and   $\|K(x,\cdot)\|_{L^{p'}(Y)} $ is uniformly bounded on $Y$.
	Then the operator $A$ with kernel $K$  maps $L^p(X)$ into $L^\infty(Y)$ and
	\begin{equation}\label{normk2}
	\|A\|_{L^p(X)\to L^\infty(Y)}=\essupp_{x\in X} \|K(x,\cdot)\|_{L^{p'}(Y)}.
	\end{equation}
\end{proposition}

Let us introduce the permutation $c_1$ of the set $\{1,\dots, 4d\}$ with related $\tilde{c}_1$ given by, for $x_i\in\rd$, $i=1,\dots,4$,
\begin{equation}\label{Dh}\tilde{c}_1(x_1,x_2, x_3,x_{4})= 
\begin{pmatrix} I_d&0_d&0_d&0_d\\0_d&0_d&I_d&0_d\\
0_d&I_d&0_d&0_d\\
0_d&0_d&0_d&I_d\end{pmatrix}\,\begin{pmatrix} x_1\\
x_{2}\\x_{3}\\  x_{4}
\end{pmatrix}=(x_1,x_3,x_2,x_4).
\end{equation}
Recall that we only consider lattices of the form $\Lambda=\a \zdd$.
\begin{theorem}\label{teo1} Suppose $p\in [1,\infty]$. 	Let $g$ be a window function in $\cS(\rd)$ such that $\mathcal{G}(g,\Lambda)$ is a frame for $L^2(\rd)$ with a dual frame $\mathcal{G}(\gamma,\Lambda)$. A linear continuous operator $A:\cS(\rd)\to \cS'(\rd)$ is bounded $M^1(\rd)\to M^p(\rd)$ if and only if its distrinution kernel $K$ satisfies 
	\begin{equation}\label{kernel1}
	V_G K\circ \tilde{c_1}\in \ell^{p,\infty}(\Lambda\times\Lambda),
	\end{equation}
	where $G=g\otimes \bar{\gamma}$.\par Moreover 
	\[
	\|A\|_{M^1\to M^p}\asymp \|V_G K\circ \tilde{c_1}\|_{\ell^{p,\infty}(\Lambda\times\Lambda)}.
	\]
	\end{theorem}
\begin{proof} Assume first \eqref{kernel1}.
For $\lambda=(\lambda_1,\lambda_2), \mu=(\mu_1,\mu_2)\in \Lambda$, observe that the Gabor matrix of the operator $A$ is given by
\begin{align}\label{aggiunta}
K_{\lambda,\mu}&:= \la A \pi(\mu)\gamma, \pi(\lambda)g\ra\nonumber\\
&=\la K, \pi(\lambda_1,\lambda_2)g\otimes \pi(\mu_1,-\mu_2)\bar{\gamma}\ra\nonumber\\
&=V_G K\circ \tilde{c}_1(\lambda_1,\lambda_2,\mu_1,-\mu_2).
\end{align}
This matrix turns out to be the integral kernel of the operator $\tilde{A}:=C_g A D_\gamma$, acting on sequences on $\Lambda$, where the coefficient operator $C_g$ and  synthesis operator $D_\gamma$ are defined in \eqref{DC0}. 
By \eqref{kernel1} and Proposition \ref{Prop5.2}, $\tilde{A}$ is a bounded operator $\ell^1(\Lambda)\to \ell^p(\Lambda)$. 
Since $C_g: M^p(\rd)\to \ell^p (\Lambda)$ and $D_\gamma: \ell^p(\Lambda)\to M^{p}(\rd)$ continuously, we have $A=D_\gamma \tilde{A} C_\gamma: M^1(\rd)\to M^p(\rd)$, with 
\[
\|A\|_{M^1\to M^p}\leq C \sup_{\mu\in\Lambda} \|K_{\cdot, \mu}\|_{\ell^p(\Lambda)}.
\]
\par
Vice-versa, consider the finite dimensional space
$$E_N =\{(a_\lambda)_{\lambda\in\Lambda}\,:\, a_\lambda=0\,\,\mbox{for}\,\, |\lambda|>N\}
$$
and the natural projection $P_N:\,\ell^\infty(\Lambda)\to E_N$. Consider the composition operator 
\begin{equation}\label{En}
P_N C_gAD_\gamma|_{E_N}: (E_N,\|\cdot\|_{\ell^1})\to(E_N,\|\cdot\|_{\ell^p}).
\end{equation} 
This operator is represented by the (finite) matrix 
$$ \{K_{\lambda,\mu}\}_{\lambda,\mu\in \Lambda, |\lambda|\leq N, |\mu|\leq N}.
$$
Using formula \eqref{normk} in Proposition \ref{Prop5.2} (we work on finite sets, so that the assumptions are trivially satisfied) we have
\begin{equation*}
\sup_{\mu\in\Lambda, |\mu|\leq N}  \|K_{\lambda,\mu}\|_{\ell^p(\Lambda, |\lambda|\leq N)}=\| P_N C_gAD_\gamma|_{E_N}\|_{\ell^1\to \ell^p}\leq \|C_g\|\|D_\gamma\| \|A\|_{M^1\to M^p}.
\end{equation*}
Since the right-hand side is independent of $N$, taking the supremum with respect to $N$, we obtain 
\[
\sup_{\mu\in\Lambda} \|K_{\cdot, \mu}\|_{\ell^p(\Lambda)}\leq C \|A\|_{M^1\to M^p}.
\]
This concludes the proof.
\end{proof}

We observe that, given a frame $\mathcal{G}(g,\Lambda)$ for $\lrd$, $\Lambda=\alpha \bZ^{2d}$, with dual frame $\mathcal{G}(\gamma,\Lambda)$, 
 then $\mathcal{G}(\bar{\gamma},\Lambda)$ is still a frame for $\lrd$
 and $\mathcal{G}(g\otimes \bar{\gamma},\Lambda\times \Lambda)$ is a frame for $\lrdd$. Using Theorem \ref{GM} we can therefore reformulate the previous result in terms of mixed modulation spaces.
 
 \begin{corollary}\label{cor1} Suppose $p\in [1,\infty]$. A linear continuous operator $A: \cS(\rd)\to \cS'(\rd)$ is bounded from $M^1(\rd)$ into $M^p(\rd)$ if and only if its distribution kernel $K$ satisfies 
 	\begin{equation}\label{kernel2}
 	 K\in M(c_1)^{p,\infty},
 	\end{equation}
 	where $c_1$ is identified with \eqref{Dh}.\par
	Moreover 
	\[
	\|A\|_{M^1\to M^p}\asymp \|K\|_{M(c_1)^{p,\infty}}.
	\]
 \end{corollary}
 
 \begin{remark}
 	When $p=\infty$, condition \eqref{kernel2} becomes $K\in M^\infty(\rdd)$ and we recapture Feichtinger's kernel Theorem \cite{gei1}.
 \end{remark}
 
 A similar characterization of bounded operators $A: M^p(\rd) \to M^{\infty}(\rd)$ can be obtained arguing as above, using Proposition \ref{Prop5.4}, \eqref{aggiunta} and Theorem \ref{GM}. Precisely,
 define the permutation $c_2$ of the set $\{1,\dots, 4d\}$ with related $\tilde{c}_2$  by
 \begin{equation}\label{Dh2}\tilde{c}_2(x_1,x_2, x_3,x_{4})= 
 \begin{pmatrix} 0_d&0_d&I_d&0_d\\I_d&0_d&0_d&0_d\\
 0_d&0_d&0_d&I_d\\
 0_d&I_d&0_d&0_d\end{pmatrix}\,\begin{pmatrix} x_1\\
 x_{2}\\x_{3}\\  x_{4}
 \end{pmatrix}=(x_3,x_1,x_4,x_2),
 \end{equation}
  for $x_i\in\rd$, $i=1,\dots,4$.
  
  \begin{theorem}\label{teo2} Suppose $p\in [1,\infty]$.  A linear continuous operator $A: \cS(\rd)\to \cS'(\rd)$ is bounded from $M^p(\rd)$ into $M^\infty(\rd)$ if and only if its distribution kernel $K$ satisfies 
  		\begin{equation}\label{kernel3}
  		K\in M(c_2)^{p',\infty},
  		\end{equation}
  		where $1/p+1/p'=1$.\par
		Moreover 
		\[
		\|A\|_{M^p\to M^\infty}\asymp \|K\|_{M(c_2)^{p',\infty}}.
		\]
		
  \end{theorem}
  As a consequence of Theorems \ref{teo1} and \ref{teo2} we have the following result.
  \begin{corollary} \label{co1} Let $A:\cS(\rd) \to \cS'(\rd)$ be an operator with distribution kernel $K\in\cS'(\rdd)$. The following conditions are equivalent:
  	\begin{itemize}
  		\item [(i)] $A: M^p(\rd)\to M^p(\rd)$, for every $1\leq p\leq\infty$.
  			\item [(ii)] $A: M^1(\rd)\to M^1(\rd)$ and  $A: M^\infty(\rd)\to M^\infty(\rd)$.
  			\item [(iii)] $K\in M(c_1)^{1,\infty}\cap M(c_2)^{1,\infty}$.
  	\end{itemize}
  \end{corollary}
  \begin{proof}
  	All the equivalences follow trivially from Theorems  \ref{teo1} and \ref{teo2} but $(ii) \Rightarrow (i)$, which follows by complex interpolation for modulation spaces \cite{F1}. 
  \end{proof}
  \section{Operators acting on $M^{p,q}$}
  Let us now consider the problem of boundedness of an integral operator on $M^{p,q}(\rd)$, $1\leq p,q\leq \infty$. We do no longer expect a characterization as in the case $p=q$, studied in the previous section (see Remark \ref{oss4} below). However sufficient conditions in the same spirit can be stated.\par 
 Consider  the permutations $c_3$ and $c_4$ of the set $\{1,\dots, 4d\}$ with related $\tilde{c}_3$, $\tilde{c}_4$ defined  by
  \begin{equation}\label{Dh3}\tilde{c}_3(x_1,x_2, x_3,x_{4})= 
  \begin{pmatrix} 0_d&I_d&0_d&0_d\\
  0_d&0_d&I_d&0_d\\
  I_d&0_d&0_d&0_d\\
  0_d&0_d&0_d&I_d\end{pmatrix}\,\begin{pmatrix} x_1\\
  x_{2}\\x_{3}\\  x_{4}
  \end{pmatrix}=(x_2,x_3,x_1,x_4),
  \end{equation}
  and 
   \begin{equation}\label{Dh4}\tilde{c}_4(x_1,x_2, x_3,x_{4})= 
   \begin{pmatrix} 
   I_d&0_d&0_d&0_d
   \\0_d&0_d&0_d&I_d\\
   0_d&I_d&0_d&0_d\\
   0_d&0_d&I_d&0_d\end{pmatrix}\,\begin{pmatrix} x_1\\
   x_{2}\\x_{3}\\  x_{4}
   \end{pmatrix}=(x_1,x_4,x_2,x_3),
   \end{equation}
  for $x_i\in\rd$, $i=1,\dots,4$.\par
The following result was proved in \cite[Proposition 5.1]{fio1} (a continuous version is contained in \cite[Proposition 2.4]{fio5}).
   \begin{proposition}\label{proschur}
  	Consider an at most countable index set $J$ and the operator defined on
  	sequences $(a_{\lambda_1,\lambda_2})$ on the set
  	$\Lambda= J\times J$
  	by
  	\[
  	(Aa)_{\lambda_1,\lambda_2}=\sum_{\mu_1,\mu_2\in J}
  	{K_{\lambda_1,\lambda_2,\mu_1,\mu_2}}a_{\mu_1,\mu_2}.
  	\]
  	(i) If $K\circ\tilde{c_3}\in\ell^{1,\infty,1,\infty}(J\times J\times J\times J)$, 
  	then $A$ is continuous on
  	$\ell^{\infty,1}(\Lambda)$. Moreover,
  	\begin{equation}\label{norm1}
  \|A\|_{\ell^{\infty,1}\to \ell^{\infty,1}}\leq\|K\circ\tilde{c_3}\|_{\ell^{1,\infty,1,\infty}}.
  	\end{equation}
  	(ii)  If $K\circ\tilde{c_4}\in\ell^{1,\infty,1,\infty}(J\times J\times J\times J)$, 
  	then $A$ is continuous on
  	$\ell^{1,\infty}(\Lambda)$, with
  	\begin{equation}\label{norm2}
  	\|A\|_{\ell^{1,\infty}\to \ell^{1,\infty}}\leq\|K\circ\tilde{c_4}\|_{\ell^{1,\infty,1,\infty}}.
  	\end{equation}
  	(iii) If $K$ satisfies assumptions (i) and (ii) and moreover  $K\circ c_0, K\in
  	\ell^{1,\infty}(\Lambda\times\Lambda)$, where the permutation $c_0$ is defined in \eqref{perminv}, then the operator $A$  is
  	continuous on
  	$\ell^{p,q}(\Lambda)$, for
  	every $1\leq p,q\leq\infty$.
  \end{proposition}
  \begin{remark}\label{oss4} 
  Observe that the reverse inequalites in \eqref{norm1}, \eqref{norm2} do not hold, even if one allows a multiplicative constant. Consider for example \eqref{norm1}, with $J=\{0,1,\ldots, N-1\}$, $N\in\bN$, where
  \[
K_{\lambda_1,\lambda_2,\mu_1,\mu_2}=F_{\lambda_2,\mu_1}=\frac{1}{\sqrt{N}}e^{-2\pi i \lambda_2\mu_1}
  \]
is the Fourier matrix.  We have 
\[
(K\circ\tilde{c_3})_{\mu_1,\lambda_1,\lambda_2,\mu_2}=K_{\lambda_1,\lambda_2,\mu_1,\mu_2}
\]
so that 
   \[
   \|K\circ\tilde{c_3}\|_{\ell^{1,\infty,1,\infty}}=\|F\|_{\ell^1}=N^{3/2}
   \]
  but
   \begin{equation}\label{oss4eq}
   \|A\|_{\ell^{\infty,1}\to \ell^{\infty,1}}\leq N,
   \end{equation}
   which blows-up, as $N\to+\infty$, at a lower rate.\par
   Let us verify \eqref{oss4eq}. We have 
   \begin{align*}
\|   \| \sum_{\mu_1,\mu_2\in J}
  	{K_{\lambda_1,\lambda_2,\mu_1,\mu_2}}a_{\mu_1,\mu_2}\|_{\ell^\infty_{\lambda_1}}\|_{\ell^1_{\lambda_2}}
	&=\|\sum_{\mu_1\in J}
  	{F_{\lambda_2,\mu_1}}\sum_{\mu_2\in J}a_{\mu_1,\mu_2}\|_{\ell^1_{\lambda_2}}\\
	& =\|\sum_{\mu_1\in J}
  	{F_{\lambda_2,\mu_1}}\tilde{a}_{\mu_1}\|_{\ell^1_{\lambda_2}}
   \end{align*}
   with $\tilde{a}_{\mu_1}:=\sum_{\mu_2\in J}a_{\mu_1,\mu_2}$. Now, using the embeddings
   \[
   \|b\|_{\ell^1(J)}\leq \sqrt{N} \|b\|_{\ell^2(J)},\qquad \|b\|_{\ell^2(J)}\leq \sqrt{N} \|b\|_{\ell^\infty(J)}
   \]
   and the fact that $F$ is a unitary transformation of $\ell^2(J)$ we get \eqref{oss4eq}.
  \end{remark} 
As a consequence, we obtain the following boundedness result.\par 
  Consider  the permutations $c_5$ and $c_6$ of the set $\{1,\dots, 4d\}$ with related $\tilde{c}_5$, $\tilde{c}_6$ defined  by
  \begin{equation}
  \label{Dh3}\tilde{c}_5= \tilde{c}_1 \tilde{c}_3=
  \begin{pmatrix} 
  0_d&I_d&0_d&0_d\\
  I_d&0_d&0_d&0_d\\
  0_d&0_d&I_d&0_d\\
  0_d&0_d&0_d&I_d\end{pmatrix}
  \end{equation}
  and 
   \begin{equation}\label{Dh4}\tilde{c}_6= \tilde{c}_1 \tilde{c}_4=
   \begin{pmatrix}
    I_d&0_d&0_d&0_d\\
    0_d&I_d&0_d&0_d\\
   0_d&0_d&0_d&I_d\\
   0_d&0_d&I_d&0_d\end{pmatrix}.
   \end{equation}
 \begin{theorem}\label{teo3}  (i) A linear continuous operator $A: \cS(\rd)\to \cS'(\rd)$ is bounded on $M^{\infty,1}(\rd)$ if its distribution kernel $K$ satisfies 
 	\begin{equation}\label{kernel4}
 	K\in M(c_5)^{1,\infty,1,\infty}.
 	\end{equation}
 	(ii) A linear continuous operator $A: \cS(\rd)\to \cS'(\rd)$ is bounded on $M^{1,\infty}(\rd)$ if its distribution kernel $K$ satisfies 
 	\begin{equation}\label{kernel5}
 	K\in M(c_6)^{1,\infty,1,\infty}.
 	\end{equation}
 \end{theorem} 
 \begin{proof}
 The result is a consequence of \eqref{aggiunta}, Proposition \ref{proschur} and Theorem \ref{GM}, by using the same arguments as in the first part of the proof of Theorem \ref{teo1}.
 \end{proof}
 
By interpolation we also obtain the following result.
\begin{corollary}\label{co2} Let $A:\cS(\rd) \to \cS'(\rd)$ be an operator with distribution kernel 
\[
K\in M(c_5)^{1,\infty,1,\infty}\cap M(c_6)^{1,\infty,1,\infty}.
\]
Then 
$
A: M^{p',p}(\rd)\to M^{p',p}(\rd)$ continuously, for every $1\leq p\leq \infty$, with $1/p+1/p'=1$.
\end{corollary}
By combining Corollaries \ref{co1} and  \ref{co2}, we finally obtain the following result.
\begin{corollary}\label{co7} Let $A:\cS(\rd) \to \cS'(\rd)$ be an operator with distribution kernel \[
K\in M(c_1)^{1,\infty}\cap M(c_2)^{1,\infty}\cap M(c_5)^{1,\infty,1,\infty}\cap M(c_6)^{1,\infty,1,\infty}.
\]
Then $A: M^{p,q}(\rd)\to M^{p,q}(\rd)$ continuously, for every $1\leq p,q\leq\infty$.
\end{corollary}
\section*{Acknowledgments} This research was partially supported by  the Gruppo
Nazionale per l'Analisi Matematica, la Probabilit\`a e le loro
Applicazioni (GNAMPA) of the Istituto Nazionale di Alta Matematica
(INdAM), Italy.

\end{document}